\documentclass[a4paper]{article} 

\usepackage[utf8]{inputenc}
\usepackage[english]{babel}
\usepackage[T1]{fontenc}

\usepackage{enumitem}

\usepackage{amsmath, amssymb, amsthm, dsfont, mathtools, stmaryrd, wasysym}
\usepackage{makecell}

\usepackage{centernot}

\usepackage{xcolor}
\usepackage{csquotes}

\usepackage{xparse,xstring}

\usepackage[hidelinks]{hyperref}

\usepackage[capitalise,nameinlink,noabbrev]{cleveref}

\usepackage{eqparbox}

\usepackage{tikz}
\usetikzlibrary{positioning}
\usetikzlibrary{shapes.geometric}
\usetikzlibrary{fit}

\theoremstyle{plain}
\newtheorem{theorem}{Theorem}[section]
\newtheorem{proposition}[theorem]{Proposition}

\newtheorem{corollary}[theorem]{Corollary}
\newtheorem{lemma}[theorem]{Lemma}

\theoremstyle{definition}
\newtheorem{definition}[theorem]{Definition}
\newtheorem{example}[theorem]{Example}

\theoremstyle{plain}
\newtheorem{observation}[theorem]{Observation}

\NewDocumentCommand{\alert}{s m}{\IfBooleanTF{#1}{\paragraph{}}{\par}\textbf{#2: }}

\RenewDocumentCommand{\phi}{}{\varphi}
\RenewDocumentCommand{\epsilon}{}{\varepsilon}

\NewDocumentCommand{\N}{}{\ensuremath{\mathbb{N}}}
\NewDocumentCommand{\Z}{}{\ensuremath{\mathbb{Z}}}

\NewDocumentCommand{\defas}{}{\vcentcolon=}

\NewDocumentCommand{\defiff}{}{\vcentcolon\mkern-5mu\iff}

\RenewDocumentCommand{\max}{}{\operatorname{max}}

\NewDocumentCommand{\set}{m o}{\left\{ #1 \IfValueT{#2}{\mid #2} \right\}}
\NewDocumentCommand{\collection}{m o}{\left( #1 \right)\IfValueT{#2}{_{#2}} }

\NewDocumentCommand{\powset}{m}{\mathcal{P}(#1)}

\RenewDocumentCommand{\subset}{}{\subseteq}
\RenewDocumentCommand{\supset}{}{\supseteq}

\NewDocumentCommand{\card}{m}{\left| #1 \right|}

\NewDocumentCommand{\isom}{}{\cong}

\NewDocumentCommand{\signature}{}{\sigma}

\NewDocumentCommand{\preds}{O{\signature}}{\mathcal{P}_{#1}}
\NewDocumentCommand{\rels}{O{\signature}}{\mathcal{R}_{#1}}
\NewDocumentCommand{\ar}{m}{\operatorname{ar}(#1)}

\NewDocumentCommand{\quantifier}{}{Q}

\NewDocumentCommand{\quantifiers}{}{\mathcal{Q}}

\NewDocumentCommand{\str}{m}{\mathfrak{#1}}

\NewDocumentCommand{\al}{}{\mathcal{L}}

\NewDocumentCommand{\kl}{o}{\mathcal{L}^{\IfValueTF{#1}{#1}{\quantifiers}}}

\NewDocumentCommand{\ML}{}{\operatorname{ML}}
\NewDocumentCommand{\FO}{}{\operatorname{FO}}

\NewDocumentCommand{\st}{o}{\operatorname{ST}\IfValueT{#1}{(#1)}}
\NewDocumentCommand{\ST}{o}{\operatorname{ST}}

\NewDocumentCommand{\Ass}{O{k} m}{\operatorname{Ass}_{#1}(#2)}

\NewDocumentCommand{\ltrue}{}{\top}
\NewDocumentCommand{\lfalse}{}{\bot}
\NewDocumentCommand{\limplies}{}{\rightarrow}
\NewDocumentCommand{\liff}{}{\leftrightarrow}

\NewDocumentCommand{\qr}{}{\operatorname{qr}}

\RenewDocumentCommand{\models}{}{\vDash}
\NewDocumentCommand{\nmodels}{}{\nvDash}

\NewDocumentCommand{\Th}{D<>{\kl} o m}{\operatorname{Th}_{#1}\IfValueT{#2}{^{#2}}(#3)}

\NewDocumentCommand{\player}{m}{\textbf{\MakeUppercase{\romannumeral #1}}}

\NewDocumentCommand{\filter}{}{\mathcal{F}}
\NewDocumentCommand{\uf}{}{\mathcal{U}}

\NewDocumentCommand{\de}{m}{\mathbf{#1}}

\NewDocumentCommand{\indices}{O{\alpha_{\uf}} m}{\llbracket #2 \rrbracket_{#1}}

\newcommand{\nc}{\newcommand}
\nc{\hnt}{\rule{0em}{3ex}}
\nc{\look}{\marginpar{$\bullet$}}
\nc{\ins}[1]{\bigskip\noindent
\framebox{\begin{minipage}{.95\textwidth} \sloppy \noindent \em #1 \end{minipage}}\bigskip}
\renewcommand{\subset}{\subseteq}
\renewcommand{\phi}{\varphi}

\newenvironment{romanenumerate}%
{\begin{list}{(\roman{enumi})}{\usecounter{enumi}
\setlength{\labelwidth}{2cm}
\setlength{\itemindent}{0pt}
\setlength{\itemsep}{0.5\itemsep}
\setlength{\topsep}{\itemsep}
\setlength{\parsep}{0pt}
}}{\end{list}}
\nc{\bre}{\begin{romanenumerate}}
\nc{\ere}{\end{romanenumerate}}
\newenvironment{alphaenumerate}%
{\begin{list}{(\alph{enumii})}{\usecounter{enumii}
\setlength{\labelwidth}{2cm}
\setlength{\itemindent}{0pt}
\setlength{\itemsep}{0.5\itemsep}
\setlength{\topsep}{\itemsep}
\setlength{\parsep}{0pt}
}}{\end{list}}
\nc{\bae}{\begin{alphaenumerate}}
\nc{\eae}{\end{alphaenumerate}}
\newenvironment{numenumerate}%
{\begin{list}{(\arabic{enumiii})}{\usecounter{enumiii}
\setlength{\labelwidth}{2cm}
\setlength{\itemindent}{0pt}
\setlength{\itemsep}{0.5\itemsep}
\setlength{\topsep}{\itemsep}
\setlength{\parsep}{0pt}
}}{\end{list}}
\nc{\bne}{\begin{numenumerate}}
\nc{\ene}{\end{numenumerate}}

\nc{\KM}{\mathcal{K}(\M)}

\nc{\INQML}{\mathrm{InqML}}
\nc{\MSO}{\mathrm{MSO}}
\nc{\LL}{\mathcal{L}}

\newcommand{\brck}[1]{[\![#1]\!]}
\usepackage{color}

\renewcommand{\preceq}{\preccurlyeq}
\renewcommand{\succeq}{\succcurlyeq}

\author{
 Janek Härtter \\
 \small\href{mailto:math@haertter.com}{math@haertter.com}
\and  Martin Otto \\ \small\href{mailto:otto@mathematik.tu-darmstadt.de}{otto@mathematik.tu-darmstadt.de}}
\date{January 2026}
\title{A Class of Generalised Quantifiers for k-Variable Logics}

\begin{document}
\maketitle
\begin{abstract}
  We introduce $k$-quantifier logics -- logics with access to $k$-tuples of elements 
and very general quantification patterns for transitions between $k$-tuples. 
The framework is very expressive and encompasses e.g.\ the
$k$-variable fragments of first-order logic, modal logic, and monotone
neighbourhood semantics. We introduce a corresponding notion of
bisimulation and prove variants of the classical Ehrenfeucht--Fra\"\i
ss\'e and Hennessy--Milner theorem. Finally, we show a Lindström-style
characterisation for $k$-quantifier logics that satisfy \L o\'s'
theorem by proving that they are the unique maximally expressive logics
that satisfy \L o\'s' theorem and are invariant under the associated bisimulation relations.
\end{abstract}

\section{Introduction}

Lindstr\"om's theorem for first-order logic $\FO$ is a hallmark of
abstract model theory. In its main version, as treated as   
Lindstr\"om's First Theorem e.g.\ in~\cite{EFT}, it characterises $\FO$
as \emph{maximally expressive} among logics that satisfy some very
basic standard properties of abstract logics together with 
\emph{compactness} and the downward \emph{L\"owenheim-Skolem
  Theorem} as the essential constraints. Among the standard properties that competitors must
satisfy are basic semantic principles 
like isomorphism invariance, closure under booleans, natural behaviour
w.r.t.\ renaming of symbols, and the like. The two crucial extra
properties that are characteristic of $\FO$ against this backdrop are its compactness
(satisfiability of every finite subset implies satisfiability)
and the countable L\"owenheim-Skolem
condition (satisfiability of a  countable set of formulae implies
satisfiability over an at most countably infinite domain).
For $\FO$, compactness is an immediate corollary to G\"odel's
Completeness Theorem, but can also be derived -- in purely
model-theoretic terms --
from compatibility of $\FO$ with ultraproducts.
This concerns the universal-algebraic generation of natural quotients of direct products 
w.r.t.\ ultrafilter equivalence.  That essential model-theoretic
feature of $\FO$ is the content of \L o\'s' Theorem (cf.~Theorem~\ref{theorem:los} below, and see e.g.~\cite{EFT} for a textbook account),
which states that
the formulae $\phi \in \FO$ satisfied in such an ultrafilter-reduced product
over an infinite family of structures are precisely those for which
the set of component structures where $\phi$ is satisfied is in the 
ultrafilter
(see Section~\ref{section:ultraproducts} for a very short review of the basic
terminology, which is standard textbook material).
Deeper set-theoretic principles are involved in the Keisler--Shelah
Theorem~\cite[Theorem~6.1.15]{ChangKeisler1990}, which links this \L o\'s-phenomenon more directly to $\FO$
expressiveness by stating that elementarily equivalent structures 
possess isomorphic ultrapowers. This implies that no logic that respects isomorphisms
and is compatible with ultraproducts (or even just ultrapowers, as a
special case)  can possibly distinguish any two structures that 
are indistinguishable in~$\FO$; in other words: any such logic must
not only satisfy compactness, but its expressive power is directly dominated by
that of $\FO$.
Building on this insight, Sgro established the maximality of logics on certain classes of models as being maximally expressive among logics satisfying \L o\'s' theorem \cite{sgro1977maximal}.
So much for the classical picture.

A much more recent rekindling of the interest in Lindstr\"om-like
results has come with an emphasis on additional semantic
constraints, mostly in the sense of invariance conditions.
These are often motivated by typical application domains where often the desirable
levels of expressiveness can be considerably weaker than, or possibly also 
incomparable with, $\FO$. One of the prominent scenarios 
in this vein is that of modal logics where
specific constraints on the accessibility of information are modelled 
in Kripke structures. Under this modelling regime, the full expressive power  
of $\FO$ would potentially capture irrelevant features while maybe  
unnecessarily missing out on others. In the basic modal scenario, 
the natural apriori intended restriction for the semantics
is captured by \emph{bisimulation invariance}. This semantic
constraint can be seen as a severe strengthening of
isomorphism invariance, which however
may arguably be no less natural in specific contexts. 
The renewed investigation of Lindstr\"om phenomena with a systematic focus on
such restricted scenarios is highlighted in the seminal paper by 
van~Benthem, ten~Cate and V\"a\"an\"anen~\cite{vBtCV09}, and also
explored e.g.\ in~\cite{OP08}. Clearly Lindstr\"om-type maximality
assertions admit many variations in terms of the choice of
background principles from abstract model theory (like compactness,
the Tarski union property TUP, a \L o\'s-property, or locality principles)
besides the additional, domain-specific semantic
constraint embodied in invariance conditions 
or restrictions to classes of models as in frame conditions for
modal logic~\cite{Enqvist}.
Depending on how much 
the background principles already on their own restrict the
expressive power of the target logic to a standard logic like $\FO$,
a Lindstr\"om result may also be read as a \emph{characterisation theorem} 
in relation to the extra invariance condition. To give one example in
the classical modal scenario targeting basic modal logic $\ML$,
van~Benthem's characterisation of $\ML$ as expressively complete for
the bisimulation-invariant fragment of $\FO$ can also be cast as a corollary
to a Lindstr\"om theorem that establishes $\ML$ as maximally
expressive among bisimulation-invariant logics satisfying some
sufficiently strong background conditions that in effect tie them to within
$\FO$, cf.~\cite{{vBtCV09},OP08,Enqvist}.
Interestingly, semantic invariance conditions based on
Ehrenfeucht--Fra\"\i ss\'e-style back-and-forth
equivalences (viz.\ partial isomorphy) can also be seen at the
root of the classical Lindstr\"om result (cf.\ presentation in~\cite{EFT}) 
and variations for more expressive logics in the spirit of abstract
model theory~\cite{FlumBarwiseFeferman}. Here we use a natural adaptation 
of bisimulation invariance on $k$-tuples as a semantic invariance condition
that directly reflects the semantic constraints in the underlying
$k$-quantifiers. Also in this respect our results highlight -- once
again -- the intimate connection between Lindstr\"om and expressive 
completeness results. In a different vein, 
recent investigations by Liao~\cite{Liao23} generalise
the classical Lindstr\"om proof to
an entire class of logics that can encode their respective
Ehrenfeucht--Fra\"\i ss\'e games.
\par The idea for the present paper originated from the first author's bachelor's thesis \cite{Janek}.
We use the \L o\'s-property of compatibility with ultraproducts
as an extremely strong background condition from abstract model
theory for a Lindström-style characterisation. While this may look like a possible overkill in light of the
Keisler--Shelah Theorem (as it imposes a first-order ceiling for
expressiveness), it stands out as a very neat principle that may be
motivated from a purely universal-algebraic point of view.
But more importantly, our main Lindstr\"om argument in
Section~\ref{section:lindstroem} 
does not rely on Keisler--Shelah but rather establishes maximality, for a
very wide class of target logics in question, by techniques that do
not directly draw on their first-order nature -- even though that is
indirectly imposed by the \L o\'s-property. And even in light of that
deeper implication, which we discuss in Section~\ref{section:ultraproducts}, and 
for a corresponding interpretation of our main result in terms of
characterisation theorems for fragments of $\FO$, the generality 
w.r.t.\ the target logics involved and their natural semantic invariance
conditions make these results interesting in our view. 
As for those target logics we focus on a rather general framework
that does not immediately suggest first-order character but rather
focuses on an intuition of limited access to, and information links
between, local configurations or states, as exemplified in familiar
examples in the range of modal logics and neighbourhood semantics~\cite{Pacuit2017}.
For relevant \emph{local configurations} we concentrate on fixed-arity
tuples of elements, and correspondingly consider structures with
$k$-tuples of elements, \emph{$k$-pointed structures}, for fixed $k
\geq 1$, as input for the semantics of our abstract logics.%
\footnote{This is similar in spirit to  the limitation to
  $k$-variable fragments of standard logics; 
  cf.\ $k$-variable fragments like $\FO^k \subset \FO$ or its
  infinitary variants, and, e.g.\ especially their extensions by
  counting quantifiers, like $\mathrm{C}^k \subset \FO$,
  with their established relevance in algorithmic model theory~\cite{EbbinghausFlumFMT,LibkinFMT,DawarLindellWeinstein95,OttoLNL}.}
The expressive power of our target logics then rests on an apriori very
general notion of \emph{$k$-quantifiers}, which access sets of
accessible $k$-tuples from any given $k$-tuple in such a structure; 
the underlying access pattern and style of quantification based on these 
is very much akin to (monotone) neighbourhood semantics if we think of
the sets of accessible $k$-tuples as neighbourhoods. The main point 
lies in the extremely liberal format for these $k$-quantifiers (or the
associated neighbourhood systems of witness sets supporting their
semantics) that admits almost any sensible (viz.~isomorphism
invariant) allocation of systems of witness sets in $k$-pointed
structures, cf.\ Section~\ref{section:klogics} for details.
The additional semantic invariance conditions that come
with a collection of $k$-quantifiers are modelled on the natural
generalisations of bisimulation invariance, for a notion of
Ehrenfeucht--Fra\"\i ss\'e style 
back\&forth games in which the alternating probing of
$k$-configurations is guided by access to the available
witness sets underlying the semantics of $k$-quantifiers, as discussed
in Sections~\ref{section:bisim},~\ref{section:EF} and~\ref{section:HMthm}. 

\par Other frameworks for capturing the semantics of entire classes of
logics have been explored in the past. Most notably,
\emph{Lindström quantifiers} (cf.~\cite[Chapter~4]{ebbinghaus1985})
are much more general and versatile 
than our $k$-quantifiers. While this expressive strength allows them to capture almost any conceivable logic, their power is
rather an obstacle in our setting.
Their very generality makes it hard to capture and restrict their expressive power
in the manner proposed here. 
\par\emph{Neighbourhood semantics} for modal logic, as explored extensively in~\cite{Pacuit2017}, is closely related to our notion of $1$-quantifier logics, the unary version of $k$-quantifier logic. There is essentially a one-to-one correspondence between up-ward-monotone neighbourhood functions for the box operator and
our witness sets for $1$-quantifiers. The two approaches mainly differ in that neighbourhood semantics considers the neighbourhood function as part of the admissible specifications of the structures under consideration, while the witness sets of $1$-quantifiers are considered as `externally determined' by the quantifier.
One could therefore think of $1$-quantifiers as imposing extremely restrictive frame conditions for neighbourhood functions.
\par For an illustration of the expressive power of $1$-quantifier logics see Examples~\ref{example:1-logics}, \ref{example:mns}, and \ref{example:fo-k} below.
\paragraph{} This paper is structured as follows: In
Section~\ref{section:klogics}, we introduce our notion of
$k$-quantifiers and $k$-quantifier logics, and give well-known examples of
logics covered by our framework.
We define a corresponding bisimulation game in
Section~\ref{section:bisim}. In Sections~\ref{section:EF} and~\ref{section:HMthm}
we adapt the classical
Ehrenfeucht--Fra\"\i ss\'e and Hennessy--Milner theorems to
$k$-quantifier logics. In Section~\ref{section:ultraproducts}, we investigate properties of $k$-quantifier logics that have the \L o\'s property and conclude our investigations with a Lindström-style theorem for such $k$-quantifier logics in Section~\ref{section:lindstroem}.

\section{Abstract $k$-Logics}
\label{section:klogics}
Throughout our investigations, we restrict ourselves to
purely relational signatures. The arity of a given relation $R$ is denoted as $\ar{R}$.
\par The universe of a $\signature$-structure $\str{A}$ is denoted as $A$. A pair $(\str{A}, \alpha)$ consisting of a $\signature$-structure $\str{A}$ and a corresponding variable assignment $\alpha$ is called a \emph{pointed $\signature$-structure} or \emph{pointed structure} for short.
\par For $k \in \N$, a \emph{$k$-assignment} is an assignment over $k$ variables. A pointed $\signature$-structure $(\str{A}, \alpha)$ is a \emph{$k$-pointed ($\signature$-)structure} if $\alpha$ is a $k$-assignment. We associate $k$-assignments $\alpha : \set{x_1, \dots, x_k} \to A$ with the $k$-tuple $(\alpha(x_1), \dots, \alpha(x_k)) \in A^k$ to ease notation.

\begin{definition}[abstract logics]\label{definition:abstract-logics}
	An \emph{abstract logic} $\al = (L, \models_{\al})$ is a pair consisting of
	\begin{itemize}
		\item a function $L$ mapping signatures to sets of formulae, and
		\item a satisfaction relation $\models_{\al}$ relating pointed $\signature$-structures and formulae in $L(\signature)$.
	\end{itemize}
	We require an abstract logic to
	\begin{itemize}
		\item 
		      be invariant under isomorphism, i.e.\ for any $\phi \in L(\signature)$, pair of isomorphic $\signature$-structures $\iota: \str{A} \simeq \str{B}$ and variable assignment $\alpha$ in $\str{A}$,
		      \begin{align*}
			      \str{A}, \alpha \models_{\al} \phi \quad & \iff \quad \str{B}, \iota \circ \alpha \models_{\al} \phi
			      ,\end{align*}
		\item be closed under boolean connectives, i.e.\
		      \begin{align*}
			      \phi, \psi \in L(\signature) \quad & \implies \quad \ltrue, \, \lnot \phi, \, (\phi \land \psi) \in L(\signature)
			      ,\end{align*}
		      \item have the (syntactic) finite occurrence
                        property, i.e.\ for every signature
                        $\signature$, $\phi \in \al(\signature)$ there
                        is a finite $\signature_0 \subset \signature$
                        s.t.\ $\phi \in \al(\signature_0)$,
         \item be invariant under renaming, i.e.\ for any
         signature $\signature$, fresh $k$-ary
         relations $R, R^{\prime} \notin \signature$,
         and $\phi \in \al(\signature \cup \set{R})$,
         there is some $\phi^{\prime} \in \al(\signature \cup
         \set{R^{\prime}})$ such that for all pointed $\signature$-structures
	$(\str{A}, \alpha)$ and $R^{\str{A}} = {R^{\prime}}^{\str{A}} \subset A^k$,
	\begin{align*}
		\str{A}, R^{\str{A}}, \alpha \models \phi \quad&\iff\quad \str{A}, {R^{\prime}}^{\str{A}}, \alpha \models \phi^{\prime}
	.\end{align*}

	\end{itemize}
	\par An \emph{abstract $k$-logic} $\al$ is an abstract logic that defines semantics for $k$-pointed $\signature$-structures with variable assignments over $\mathcal{X}_k \defas \set{x_1, \dots, x_k}$.
\end{definition}

We treat the connectives $\lor$, $\limplies$ and $\liff$ as well as the formula $\lfalse$ as abbreviations in the usual manner.
To ease notation, we furthermore
do not distinguish between $L$ and
$\al$, or write just $\models$ instead of $\models_{\al}$, where this causes no confusion.
\par Note that as we mostly consider compact logics throughout our investigations, the finite occurrence property is mainly stated for convenience as compactness and invariance under renaming implies the finite occurrence property in a semantic sense (cf.\ \cite[Proposition~5.1.2]{ebbinghaus1985}).

\paragraph{} An abstract logic $\al_1$ is a \emph{fragment} of an
abstract logic $\al_2$ or $\al_2$ an \emph{extension} of $\al_1$,
denoted as ${\al_1 \preceq \al_2}$ or ${\al_2 \succeq \al_1}$,
if for every signature $\signature$ and every formula
$\phi \in \al_1(\signature)$, there exists a formula $\psi \in \al_2(\signature)$ such that for all pointed $\signature$-structures $(\str{A}, \alpha)$,
\begin{align*}
	\str{A}, \alpha \models \phi \quad & \iff \quad \str{A}, \alpha \models \psi
.\end{align*}

As syntactic differences are immaterial for our purposes, we
assume that for all signatures $\signature$,
\begin{align*}
	\al_1 \preceq \al_2 \quad & \implies \quad L_1(\signature) \subset L_2(\signature)
.\end{align*}
\par Two abstract logics $\al_1$ and $\al_2$ are \emph{equivalent}, denoted as $\al_1 \equiv \al_2$, if $\al_1 \preceq \al_2$ and $\al_2 \preceq \al_1$.

\begin{definition}[$k$-quantifiers]\label{definition:k-quantifiers}
	A \emph{$k$-quantifier} $\quantifier$ is a symbol together with
	\begin{itemize}
		\item an associated signature $\signature_{\quantifier}$, and
		\item a \emph{witness set} $\quantifier(\str{A}, \alpha) \subset \powset{A^k}$
                  for every (expansion of a)
                  $k$-pointed $\signature_{\quantifier}$-structure
                  $(\str{A}, \alpha)$. $\quantifier$ needs to be
                  well-behaved w.r.t.\ isomorphisms and
expansions, i.e.\ if $\iota: (\str{A} \restriction \signature_{\quantifier}) \simeq
                      (\str{B} \restriction \signature_{\quantifier})$,
			then for any $k$-assignment $\alpha$
		      \begin{align*}
			      \quantifier(\str{B} \restriction
                        \signature_{\quantifier}, \iota \circ \alpha) =
                        \quantifier(\str{B}, \iota \circ \alpha) = \set{
                        \set{ \iota \circ \gamma }[ \gamma \in s ] }[ s \in
                        \quantifier(\str{A} \restriction
                        \signature_{\quantifier}, \alpha) ]
                      .\end{align*}
	\end{itemize}
\end{definition}

Note that the provision that $k$-quantifiers are well-behaved w.r.t.\ expansions already implies that $k$-quantifiers respect isomorphisms.
\par We call the members of a witness set \emph{witnesses} and, as the underlying structure should usually be clear from context, we may shorten $\quantifier(\str{A}, \alpha)$ to just $\quantifier(\alpha)$.
\par For a given set  $\quantifiers$ of $k$-quantifiers, we let
\begin{align*}
  \quantifiers_{\signature} & \defas \set{\quantifier \in \quantifiers}[\signature_{\quantifier} \subset \signature]
.\end{align*}

\begin{definition}[$k$-quantifier logics]\label{definition:k-logics}
	A \emph{$k$-quantifier logic} $\kl$ is an abstract
	$k$-logic based on 
	a class of $k$-quantifiers $\quantifiers$ with syntax and
        semantics given as follows.
        \par For a given signature $\signature$, the \emph{syntax} of $\kl(\signature)$ is given inductively via
	\begin{itemize}
               \item $\ltrue \in \kl(\signature)$,
		\item for all $l$-ary relations $R \in \signature$ and
                  variables $y_{1}, \dots, y_{l} \in \mathcal{X}_{k}$,
                  the formula $R(y_{1}, \dots, y_{l})$ is in $\kl(\signature)$,
		\item for all $\phi \in \kl(\signature)$, also $\lnot \phi \in \kl(\signature)$,
		\item for all $\phi_1, \phi_2 \in \kl(\signature)$, also $(\phi_1 \land \phi_2) \in \kl(\signature)$,
		\item for all $\phi \in \kl(\signature)$ and $\quantifier \in \quantifiers_{\signature}$, also $\quantifier \phi \in \kl(\signature)$.
	\end{itemize}
	\par The \emph{satisfaction relation} is also defined in an inductive fashion: for all signatures $\signature$, $\quantifier \in \quantifiers_{\signature}$ and $k$-pointed $\signature$-structures $(\str{A}, \alpha)$,
	\begin{itemize}
		\item $\str{A}, \alpha \models \ltrue$,
		\item $\str{A}, \alpha \models R(y_{1}, \dots, y_{l})$ if $(\alpha(y_{1}), \dots, \alpha(y_{l})) \in R^{\str{A}}$,
		\item $\str{A}, \alpha \models \lnot \phi$ if $\str{A}, \alpha \nmodels \phi$,
		\item $\str{A}, \alpha \models (\phi_1 \land \phi_2)$ if $\str{A}, \alpha \models \phi_1$ and $\str{A}, \alpha \models \phi_2$,
		\item $\str{A}, \alpha \models \quantifier \phi$ if there exists an $s \in \quantifier(\str{A},\alpha)$ such that for all $\gamma \in s$: $\str{A}, \gamma \models \phi$.
	\end{itemize}
        Inspired by flat team semantics, we define for subsets $s \subset A^k$ that 
	\begin{align*}
		\str{A}, s \models \phi \qquad & \defiff \qquad \text{f.a.\ } \alpha \in s: \str{A}, \alpha \models \phi
		.\end{align*}
\end{definition}

Note that while we could include equality as a literal in the
definiton of $k$-quantifier logics, we chose not to as it can be
easily added by introducing it as an extra binary relation if desired.

\begin{observation}[upward monotonicity]
	Let $\kl$ be a $k$-quantifier logic. Then $\kl$ is \emph{upward monotone}, i.e., for any formula $\phi \in \kl(\signature)$, $k$-pointed $\signature$-structure $(\str{A}, \alpha)$, and $s \in Q(\alpha)^{\uparrow} \defas \set{s^{\prime} \subset A^k}[\text{ex.\ } t \in Q(\alpha) \text{ with } t \subset s^{\prime}]$,
	\begin{align*}
		\str{A}, s \models \phi \quad &\implies \quad \str{A}, \alpha \models Q \phi
	.\end{align*}
\end{observation}

\begin{definition}[quantifier rank]\label{definition:quantifier-rank}
	For a given $k$-quantifier logic $\kl$ and signature $\signature$, we define the \emph{quantifier rank} of formulae in $\kl(\signature)$ inductively as
	\begin{align*}
		\qr(\ltrue)              & \defas 0,                                                                                                                     \\
		\qr(R(y_1, \dots, y_l))  & \defas 0                                & \text{for any } R \in \signature, \text{variables } y_1, \dots, y_l,                     \\
		\qr(\lnot \phi)          & \defas \qr(\phi)                        & \text{for any } \phi \in \kl(\signature),                                           \\
		\qr(\phi_1 \land \phi_2) & \defas \max\{\qr(\phi_1), \qr(\phi_2)\} & \text{for any } \phi_1, \phi_2 \in \kl(\signature),                                 \\
		\qr(\quantifier \phi)    & \defas 1 + \qr(\phi)                    & \text{for any } \phi \in \kl(\signature), \quantifier \in \quantifiers_{\signature}
	.\end{align*}
\end{definition}

\begin{definition}[$\kl$-elementary equivalence]\label{definition:elementary-equivalence}
	Let $\kl$ be a $k$-quantifier logic. Then two $k$-pointed $\signature$-structures $(\str{A}, \alpha)$ and $(\str{B}, \beta)$ are \emph{$\kl$-elementarily equivalent}, denoted as $(\str{A}, \alpha) \equiv_{\kl} (\str{B}, \beta)$, if for all $\phi \in \kl(\signature)$,
	\begin{align*}
		\str{A}, \alpha \models \phi \quad & \iff \quad \str{B}, \beta \models \phi
		.\end{align*}
	\par Similarly, $(\str{A}, \alpha)$ and $(\str{B}, \beta)$ are \emph{$\kl_q$-elementarily equivalent} for $q \in \N$, denoted $(\str{A}, \alpha) \equiv^q_{\kl} (\str{B}, \beta)$, if for all $\phi \in \kl(\signature)$ with $\qr(\phi) \leq q$,
	\begin{align*}
		\str{A}, \alpha \models \phi \quad & \iff \quad \str{B}, \beta \models \phi
		.\end{align*}
\end{definition}

\begin{definition}[$\mathcal{L}$-theories]\label{definition:theory}
  For a given k-quantifier logic $\kl$ and signature $\signature$, we denote the \emph{$\kl$-theory} of a pointed $\signature$-structure $(\str{A}, \alpha)$ as
  \begin{align*}
    \Th{\str{A}, \alpha}    & \defas \set{\phi \in \kl(\signature)}[\str{A}, \alpha \models \phi]                    \\
    \intertext{and the $\kl_q$-theory for a given quantifier rank $q \in \N$ as}
    \Th[q]{\str{A}, \alpha} & \defas \set{\phi \in \kl(\signature)}[\str{A}, \alpha \models \phi, \qr(\phi) \leq q]. \\
    \intertext{Similarly, for $s \subset A^k$, we define}
    \Th{\str{A}, s}         & \defas \set{\phi \in \kl(\signature)}[\str{A}, s \models \phi],                        \\
    \Th[q]{\str{A}, s}      & \defas \set{\phi \in \kl(\signature)}[\str{A}, s \models \phi, \qr(\phi) \leq q]
  .\end{align*}
\end{definition}

\paragraph{}

\begin{example}[modal quantifiers as 1-quantifiers]\label{example:1-logics}
For a given binary relation $R \subset A \times A$, we denote the set of successors of a given element $a \in A$ as
\begin{align*}
	R[a] \defas \set{c \in A}[(a, c) \in R]
	.\end{align*}
	\par We start by defining some well-known 1-quantifiers that
come up in a modal setting,
as summarised in the following table in which $R$ denotes a binary
relation.
	\begin{center}
		\begin{tabular}{c | c | c | c}
			$Q$                 & semantics
                  &
                    $\sigma_Q$
                  & reference            \\
			\hline
			$\Diamond$          & there is an $R$-successor s.t.\                      & $(R)$       & \cite{blackburn2007} \\
			$\Diamond^{\geq k}$ & there are at least $k$ $R$-successors s.t.\                   & $(R)$       & \cite{hoek1992}      \\
			$\forall$           & all elements satisfy                                   & $\emptyset$ & \cite{goranko1992}   \\
			$\exists$           & there is an element s.t.\                            & $\emptyset$ & \cite{goranko1992}   \\
			$\circ$             & there is a cycle of
                                              length $\geq 3$ of
                                              elements s.t.\            & $(R)$       &                      \\
			$\bullet$           & there are infinitely many reflexive $R$-successors s.t.\ & $(R)$       & \cite{vBtCV09}   \\
			$\hookrightarrow$   & there is a reachable element s.t.\                   & $(R)$       & \cite{rescher2012}   \\
		\end{tabular}
              \end{center}
	These are obtained as $1$-quantifiers with the following witness sets:
	\begin{align*}
		\Diamond(\str{A}, a)          & \defas \set{ \set{c} }[ c \in R^{\str{A}}[a] ],                                                                                                                               \\
		\Diamond^{\geq k}(\str{A}, a) & \defas \set{ s \subset A }[ \card{s \cap R^{\str{A}}[a]} \geq k ],                                                                                                            \\
		\forall(\str{A}, a)           & \defas \set{ A },                                                                                                                                                             \\
		\exists(\str{A}, a)           & \defas \set{ s \subset A }[ s \neq \emptyset ],                                                                                                                               \\
		\circ(\str{A}, a)             & \defas \bigl\{ \set{c_0, \dots, c_{l - 1}} \subset A \mid l \geq 3 \text{ s.t.\ f.a.\ } i \in \Z_l: (c_i, c_{i + 1}) \in R^{\str{A}} \bigr\},      \\
		\bullet(\str{A}, a)           & \defas \set{ s \subset R^{\str{A}}[a] }[ \card{s} = \infty, \text{ f.a.\ } c \in s: (c, c) \in R^{\str{A}} ],                                                                 \\
		\hookrightarrow(\str{A}, a)   & \defas \bigl\{ \set{c_l} \subset A \mid l \geq 0, c_0, \dots, c_{l} \in A                                                                                              \\ & \qquad \qquad \text{ s.t.\ } c_0 = a \text{ and f.a.\ } 0 \leq i < l: (c_i, c_{i + 1}) \in R^{\str{A}} \bigr\}
		.\end{align*}
\end{example}

\begin{example}[monotonic neighbourhood semantics versus $1$-quantifiers]\label{example:mns}
	More generally, we can capture monotonic neighbourhood semantics in terms of $1$-quantifier logics.
	\par A neighbourhood model $\str{M} \defas (W, N, V)$ consists of a universe $W$, a neighbourhood function $N : W \to \powset{\powset{W}}$, and a valuation $V : p_i
		\mapsto P_i \in \powset{W}$, defining the relational interpretation for the basic propositions $p_i$.
	The standard logic is modelled on basic modal language, with standard semantics
	for atomic formulae, propositional connectives, and 
	\begin{align*}
		\str{M}, w \models \Box \phi \quad &\defiff \quad \set{u \in W}[\str{M}, u \models \phi] \in N(w).\end{align*}
	
	For a thorough introduction to neighbourhood semantics, we refer the reader to~\cite{Pacuit2017}.
	A neighbourhood model is monotonic if all $N(w)$ are closed under passage to supersets: $N(w) = N(w)^\uparrow$.
	Let $\mathcal{C}$ be a class of monotonic neighbourhood models such that for every pair ${(W, N, V), (W^{\prime}, N^{\prime}, V^{\prime}) \in \mathcal{C}}$
	\begin{align}
		(W,V) \simeq  (W^{\prime}, V^{\prime})
		\quad \implies \quad
		(W,V,N) \simeq  (W^{\prime}, V^{\prime},N^{\prime}) \label{nh-provisio}.\end{align}
	This allows us to inductively translate any given $\phi \in \ML$ into the language of \mbox{$1$-quantifier} logics akin to the standard translation of basic modal logic over Kripke frames. We translate box as
	$\st(\Box \psi) \defas Q \st(\psi)$
	with witness sets according to 
	$Q((W, V), w) \defas N(w)$. 
	We find that, for $P_i^{\str{M}} \defas V(p_i)$ and $w \in W$,
	\begin{align*}
		(W, N, V), w \models \phi \quad & \iff \quad (W, (P_i^{\str{M}})_{i \in \N}), w \models_{\kl[\set{Q}]} \st(\phi)
		.\end{align*}
		In this sense, $1$-quantifiers are equivalent to neighbourhood semantics in restriction to a class of models given by the frame condition above.
	\par Note that this construction `externalises' the neighbourhood functions
	to encode them as $1$-quantifiers. Thus, \eqref{nh-provisio} is necessary to avoid the violation of isomorphism invariance of $\kl[\set{Q}]$.
\end{example}

\begin{example}[first-order and counting quantifiers]\label{example:fo-k}
	Fix some $k \geq 1$ and define
	\begin{align*}
		\quantifiers \defas \set{ \exists^{\geq n}_{x_i} }[n \in \N, 1 \leq i \leq k]
	\end{align*}
	with
	$  \exists^{\geq n}_{x_i} (\str{A}, \alpha) \defas \set{ s \subset A^k }[ |s| \geq n, \text{ f.a.\ } \gamma \in s, j \neq i: \gamma(x_j) = \alpha(x_j) ]$.
	Then the logic $\kl \equiv \mathrm{C}^k$ is the $k$-variable fragment of first-order logic with counting quantifiers;
	and the restriction to $(\exists^{\geq 1}_{x_i})_{1 \leq i
          \leq k}$ corresponds to the $k$-variable fragment $\FO^k$ of
        first-order logic~\cite{EbbinghausFlumFMT,LibkinFMT,DawarLindellWeinstein95,OttoLNL}.
      \end{example}

\section{Bisimulation}\label{section:bisim}

We aim for natural characterisations of the expressiveness of $k$-quantifier logics in the traditional model-theoretic tradition of back\&forth techniques that capture levels
of $\kl$-indistinguishability between structures with
assignments. For classical first-order logic $\FO$ this is
exemplified in the original Ehrenfeucht--Fra\"\i ss\'e approach
(cf.~\cite{EFT}), which has many natural restrictions and adaptations for $\FO^k$ and $\mathrm{C}^k$, as well as, e.g.\ the guarded fragment
$\mathrm{GF}$ of $\FO$, in terms of suitable pebble game equivalences.
\par In the modal tradition, the same idea arises -- more fundamentally, in a way -- in the traditional shape of bisimulation games and equivalences. Regarding the concept of back\&forth equivalences in the bisimulation format as the most fundamental incarnation of the idea, we use that terminology but, for reasons of historical priority, refer to the crucial links as (adaptations of the) Ehrenfeucht--Fra\"\i ss\'e theorem. For presentational purposes, we also put the natural game intuition, rather than the more extensional formalisation in terms of back\&forth systems, in the foreground. 
Our adaptation of bisimulation
notions for $k$-quantifier logics is also in line with the natural adaptation
of basic modal bisimulation to monotone neighbourhood logics (cf.~\cite{Pacuit2017}),  or inquisitive modal logic (cf.~\cite{CiardelliOtto21}).
\begin{definition}[$\kl$-bisimulation]\label{definition:bisimulation-game}
	There are two variants of the \emph{$\kl$-bisimulation game}
        -- the unbounded game $G^{\quantifiers}$ and the $q$-round game $G_q^{\quantifiers}$. Both games are played by two players \player{1} and \player{2} over two $\signature$-structures $\str{A}$ and $\str{B}$. Positions in the game are
	pairs of $k$-pointed $\signature$-structures
	$(\str{A}, \alpha; \str{B}, \beta)$ with $\alpha \in A^k$ and $\beta \in B^k$.
        
	Player~\player{2} loses in any position $(\str{A}, \alpha; \str{B}, \beta)$ that violates \emph{atom equivalence}, i.e.\ in which $(\str{A}, \alpha)$ and $(\str{B}, \beta)$ can be distinguished by a quantifier-free formula. 
	\par A round, from position $(\str{A}, \alpha; \str{B},
        \beta)$, is played as follows:
	\begin{itemize}
		\item \player{1} picks one of the two pointed structures, say $(\str{A}, \alpha)$, a quantifier $\quantifier \in \quantifiers_{\signature}$, and a witness set $s \in \quantifier(\str{A}, \alpha)$.
		\item \player{2} must respond with a witness set $t \in \quantifier(\str{B}, \beta)$.
		\item \player{1} must pick some assignment $\delta \in t$.
		\item \player{2} must respond with some assignment $\gamma \in s$.
		\item The new position in the game is $(\str{A}, \gamma; \str{B}, \delta)$,
	\end{itemize}
	and either player loses during this round when stuck for a required response. 
\par We say that \player{2} wins
a play of $G^{\quantifiers}$ if \player{1} loses or if the
play continues indefinitely without loss for \player{2}.
Similarly, \player{2} wins in $G_q^{\quantifiers}$ if player
\player{1} loses or neither player loses during all $q$ rounds of the
game.
\par We say that the pointed structures $(\str{A}, \alpha)$ and
$(\str{B}, \beta)$ are \emph{$\kl$-bisimilar}, denoted
$(\str{A}, \alpha) \sim_{\kl} (\str{B}, \beta)$, if
\player{2} has a winning strategy for
$G^{\quantifiers}$ with initial configuration $(\str{A}, \alpha; \str{B}, \beta)$.
	\par Similarly, we say that the $k$-pointed structures are
        \emph{$\kl_q$-bisimilar} for some $q \in \N$, denoted
        $(\str{A}, \alpha) \sim_{\kl}^q (\str{B}, \beta)$, if
        \player{2}
has a winning strategy for
$G_q^{\quantifiers}$ from this initial position.
\end{definition}

As the $\kl$-bisimulation game is a Borel game, it is determined~\cite{Martin1975}
in the sense that
any position of the game admits a winning strategy for precisely one of the two players.

\section{An Ehrenfeucht--Fraïss\'e Theorem}
\label{section:EF}

The Ehrenfeucht--Fra\"\i ss\'e theorem for first-order logic $\FO$
essentially states a correspondence between indistinguishability by
formulae up to a given quantifier-rank $q$ and a winning strategy for \player{2} in the $q$-round $\FO$ pebble game. The  theorem carries over to basic modal logic and, as we will see in this section, also to k-quantifier logics with the corresponding notion of $\kl_q$-bisimilarity.

\begin{theorem}[Ehrenfeucht--Fra\"\i ss\'e theorem for $\kl$] \label{theorem:ehrenfeucht-fraisse}
	Let $\signature$ be a finite signature
	and $\LL :=\kl$ a $k$-quantifier logic with finitely many quantifiers in $\quantifiers_{\signature}$. For any $q \in \N$ and pair of pointed $\signature$-structures $(\str{A}, \alpha)$, $(\str{B}, \beta)$, the following are equivalent:
	\begin{enumerate}[label=(\roman*)]
		\item\label{theorem:ehrenfeucht-friasse:i} $(\str{A}, \alpha) \sim_{\LL}^q (\str{B}, \beta)$,
		\item\label{theorem:ehrenfeucht-friasse:ii}
                   $(\str{A}, \alpha) \equiv_\LL^q (\str{B}, \beta)$.
	\end{enumerate}
\end{theorem}

The proof of the theorem is a simple adaptation of the classical Ehrenfeucht--Fra\"\i ss\'e argument.

\begin{proof}[Proof by induction on $q$.]
	The base case for $q = 0$ is clear.
It remains to show the induction step. Let $q > 0$.
\par We show the implication
        \mbox{\enquote{\ref{theorem:ehrenfeucht-friasse:i} $\implies$
            \ref{theorem:ehrenfeucht-friasse:ii}}} by syntactic
        induction.
As the claim is preserved under boolean connectives, we just need to consider
$\phi = \quantifier \psi$ with $\qr(\phi) = q$.  Supposing that, e.g.\ 
$\str{A}, \alpha \models \phi$, we need to show that also $\str{B}, \beta \models \phi$. Let $s \in \quantifier(\alpha)$ be a witness for $\str{A}, \alpha \models \quantifier \psi$ and let \player{1} pick $s$ as challenge in the $\kl$-bisimulation game.
\par $\sim_\LL^q$ guarantees that \player{2} has  
        a response $t \in \quantifier(b)$ in the bisimulation game such that
	every $\delta \in t$ is $\kl_{q - 1}$-bisimilar to some $\gamma_{\delta} \in s$. By the induction hypothesis, $\str{B}, \delta \models \psi$ iff $\str{A}, \gamma_{\delta} \models \psi$ as $\qr(\psi) = q - 1$. As $s$ witnesses $\quantifier \psi$ also $\str{B}, \delta \models \psi$ for all $\delta \in t$, ensuring that $t$ is a witness for $\str{B}, \beta \models \quantifier \psi$.
	\paragraph{} For the converse implication \enquote{\ref{theorem:ehrenfeucht-friasse:ii} $\implies$ \ref{theorem:ehrenfeucht-friasse:i}}, suppose $\Th[q]{\str{A}, \alpha} = \Th[q]{\str{B}, \beta}$. W.l.o.g.\ assume \player{1} picks some $s \in \quantifier(\alpha)$. We need to find a suitable response for \player{2}.
	\par Since there are only finitely many relations, variables and quantifiers, the set of formulae up to quantifier rank $q$ is finite up to logical equivalence. In particular, the conjunctions and disjunction in
	\begin{align*}
		\chi & \defas \quantifier \bigvee_{\gamma \in s} \bigwedge \Th[q - 1]{\str{A}, \gamma}
	\end{align*}
	are effectively finite. By construction, $\str{A}, \alpha \models \chi$ is witnessed by~$s$. Since ${\Th[q]{\str{A}, \alpha} = \Th[q]{\str{B}, \beta}}$, also $\str{B}, \beta \models \chi$. Take a witness ${t \in \quantifier_{\str{B}}(\beta)}$ for $\chi$. Let $t$ be \player{2}'s response in the game and $\delta \in t$ be \player{1}'s challenge. By choice of the witness $t$, $\str{B}, \delta \models \bigvee_{\gamma \in s} \bigwedge \Th[q - 1]{\str{A}, \gamma}$. Thus, there exists some $\gamma \in s$ with $\str{B}, \delta \models \Th[q - 1]{\str{A}, \gamma}$ for \player{2} to play.\qedhere
\end{proof}

Note that the implication \enquote{\ref{theorem:ehrenfeucht-friasse:i}
  $\implies$ \ref{theorem:ehrenfeucht-friasse:ii}} even holds for
infinitely many relations and quantifiers, while the finiteness
conditions are crucial for the converse direction due to the following observation that is easily shown by induction on~$q$.

\begin{observation}\label{observation:finite-index}
	Let $\LL := \kl$ be a $k$-quantifier logic with finite
        $\quantifiers$ and $\signature$ a finite signature. Then the
        equivalence relations $\equiv_{\LL}^q$, $\sim_\LL^q$ are of finite index.
\end{observation}

\begin{corollary}[$\kl$-bisimulation invariance]\label{corollary:bisimulation-invariance}
	For a $k$-quantifier logic $\kl$ and signature $\signature$, let $(\str{A}, \alpha)$ and $(\str{B}, \beta)$ be $\kl$-bisimilar $k$-pointed $\signature$-structures. Then
	\begin{align*}
		\str{A}, \alpha \equiv_{\kl} \str{B}, \beta, \; \mbox{ or, equivalently, } \;
		\Th{\str{A}, \alpha} = \Th{\str{B}, \beta}
		.\end{align*}
	\begin{proof}
		This follows immediately from Theorem~\ref{theorem:ehrenfeucht-fraisse} as $\kl$-bisimilarity implies $\kl_q$-bisimilarity for all $q$, and as each individual  $\kl$-formula refers to just finitely many relation symbols and quantifiers, and is of finite quantifier rank. 
	\end{proof}
\end{corollary}

\paragraph{} The \emph{characteristic formulae} $\chi_{\str{A},
  \alpha}^q$, for $k$-pointed $\signature$-structures $(\str{A},
\alpha)$ and quan\-ti\-fi\-er-rank $q$, are built up inductively so as to capture the challenge-response requirements
that enable player~\player{2} to stay abreast of player~\player{1} for so many
further rounds. These characteristic formulae are such that 
\begin{equation}
	\label{chareq}
	\str{B}, \beta \models \chi_{\str{A}, \alpha}^q
	\quad \Longleftrightarrow\quad
        (\str{A}, \alpha)  \sim_{\kl}^q (\str{B}, \beta).
\end{equation}

\par The quantifier-rank~$0$ formula $\chi_{\str{A}, \alpha}^0$ just
specifies the atomic type of the $k$-assignment $\alpha$ in $\str{A}$ 
(by a finite conjunction of relational atoms and negated atoms), 
as a guarantee of atom equivalence on the righthand side of~(\ref{chareq}). 

For $q > 0$, a characteristic formula is the conjunction of
\begin{itemize}
	\item the quantifier-rank $0$ formula $\chi_{\str{A}, \alpha}^0$ in order to ensure
	      atom equivalence in~(\ref{chareq}), 
        \item conditions ensuring the existence of 
         \emph{forth responses}
         for \player{2} to all potential challenges by \player{1} played in $\str{A}$, and
	\item conditions ensuring the existence of \emph{back responses}
for \player{2} to all potential challenges by \player{1} played in $\str{B}$
	      in the context of~(\ref{chareq}).
\end{itemize}

Thus, the characteristic formula for quantifier-rank $q + 1$ is
\begin{align*}
	\chi_{\str{A}, \alpha}^{q + 1} & \defas \chi_{\str{A}, \alpha}^{0} \land \bigwedge_{\quantifier \in \quantifiers_{\signature}} \underbrace{ \Bigl( \bigwedge_{s \in \quantifier(\alpha)} \quantifier \bigvee_{\gamma \in s} \chi_{\str{A}, \gamma}^{q} \Bigr) }_{\text{\enquote{forth responses}}} \\
	                               & \qquad \land \underbrace{ \bigwedge \set{\lnot \quantifier \bigvee \Phi}[\Phi \subset \Delta_q, \str{A}, \alpha \models \neg Q \bigvee \Phi] }_{\text{\enquote{back responses}}}
\end{align*}
where $\Delta_q \defas \{ \chi_{\str{B},\beta}^q \mid (\str{B},\beta) \text{
	$k$-pointed $\signature$-structure} \}$ is the set of all characteristic
formulae at level~$q$.

\begin{proposition}[characteristic formulae]\label{proposition:characteristic-formulae}
	For a finite signature $\signature$, k-quantifier logic $\kl$ with finite $\quantifiers_{\signature}$, and $k$-pointed $\signature$-structures $(\str{A}, \alpha)$ and $(\str{B}, \beta)$,
	\begin{align*}
		\str{B}, \beta \models \chi_{\str{A}, \alpha}^q
		\quad \iff\quad (\str{B}, \beta)  \sim_{\kl}^q (\str{A}, \alpha)
		.\end{align*}
\end{proposition}

\begin{proof}
We only show the forward implication as the converse direction is obvious by
Theorem~\ref{theorem:ehrenfeucht-fraisse} since
$\str{A},\alpha \models \chi_{\str{A}, \alpha}^q$. 
Our strategy is to show by induction on~$q$ that $\str{B},\beta \models
		\chi_{\str{A},\alpha}^q$ implies that \player{2} has a winning strategy for $q$ rounds in the
	game from $(\str{A},\alpha;\str{B},\beta)$. This is obvious for $q=0$ (the
	base case).
	\par Assuming $\str{B},\beta \models
		\chi_{\str{A},\alpha}^{q+1}$, first consider a first round in which
	\player{1} first chooses some $s \in \quantifier(\str{A},\alpha)$. The
	\emph{forth}-part in $\chi_{\str{A},\alpha}^{q+1}$ ensures that $\str{B},\beta \models
		\quantifier \bigvee_{\gamma \in s} \chi_{\str{A},\gamma}^q$. Thus, we can let
	\player{2} pick a witness $t \in \quantifier(\str{B},\beta)$ such
	that $\str{B},t \models  \bigvee_{\gamma \in s} \chi_{\str{A},\gamma}^q$,
	which clearly allows \player{2} a suitable response to any challenge
	$d \in t$ that \player{1} may select.
	\par If \player{1} first chooses some $t \in \quantifier(\str{B},\beta)$, on
	the other hand, we look to the
	\emph{back}-part in $\chi_{\str{A},\alpha}^{q+1}$:
	let $\Phi := \{ \chi_{\str{B},\delta}^q \mid \delta \in t \}$ so that
	$\str{B},\beta \models \quantifier \bigvee \Phi$; it follows that
	$\str{A},\alpha \models \quantifier \bigvee \Phi$, too, since otherwise 
	$\neg Q \bigvee \Phi$ would be  a conjunct of
	$\chi_{\str{A},\alpha}^{q+1}$. If \player{2} chooses a witness $s \in
		\quantifier(\str{A},\alpha)$ for  $\str{A},\alpha \models \quantifier \bigvee
		\Phi$, then \player{2} has safe responses to challenges by \player{1}
	in the second phase of this round. 
\end{proof}

The existence of characteristic formulae directly implies a normal form according to
\begin{align*}
	\phi \equiv \bigvee \bigl\{   \chi_{\str{A}, \alpha}^{\qr(\phi)} \mid
	\str{A}, \alpha \models \phi \bigr\}.
\end{align*}
Note that the disjunction on the right is finite up to logical equivalence due to Observation~\ref{observation:finite-index}.

\section{A Hennessy--Milner Theorem}
\label{section:HMthm}
We want to find sufficient conditions for the converse of
the implication in Corollary~\ref{corollary:bisimulation-invariance}. 
The classical Hennessy--Milner theorem for basic modal logic $\ML$ states that 
$\ML$-equivalence between pointed Kripke structures implies their bisimilarity 
if the underlying frames are finitely branching; this can be substantially generalised
to the class of all $\ML$-saturated structures~(cf.~\cite{goranko2007}).
Furthermore, the proof has been adapted to neighbourhood semantics in~{\cite[Chapter~4]{Hansen09}}.
\par We here adapt this result to $k$-quantifier logics by showing that
$\kl$-equivalent $k$-pointed $\signature$-structures that are 
$\kl$-saturated (as in the following definition)
are $\kl$-bisimilar.

\begin{definition}[$\quantifier$-type]\label{definition:type}
	Let $\kl$ be a k-quantifier logic and $\str{A}$ a
        \mbox{$\signature$-structure}. For every $\quantifier \in
        \quantifiers$ and $\alpha \in A^k$, we call $\Psi \subset
        \kl(\signature)$ a (partial) \emph{$\quantifier$-type} of
        $(\str{A}, \alpha)$ if ${\str{A}, \alpha \models \quantifier
          \bigwedge \Psi_0}$ for all finite $\Psi_0 \subset \Psi$. We
        say that $\Psi$ is \emph{realised} by
        $s \in \quantifier(\alpha)$
        if $\str{A}, s \models \Psi$ (with flat semantics as introduced in Definition~\ref{definition:k-logics}).
	\par Similarly, $\Psi \subset \kl(\signature)$ is an \emph{$s$-type} of a given $s \in \quantifier(\alpha)$ if for all finite $\Psi_0 \subset \Psi$ there is a $\gamma \in s$ with $\str{A}, \gamma \models \Psi_0$. We say that $\Psi$ is \emph{realised} by $\gamma \in s$ if $\str{A}, \gamma \models \Psi$.
	\par We call $\str{A}$ \emph{$\kl$-saturated} if for all
        $\quantifier \in \quantifiers$ and $\alpha \in A^k$ every $\quantifier$-type of $(\str{A},
        \alpha)$ and every $s$-type of every $s \in Q(\alpha)$ is realised.
\end{definition}

We postpone the construction of such saturated structures to 
Section~\ref{section:lindstroem}, where we find saturated extensions for $k$-quantifier logics that are compatible with ultrapowers (Lemma~\ref{lemma:finding-saturated-extensions}).

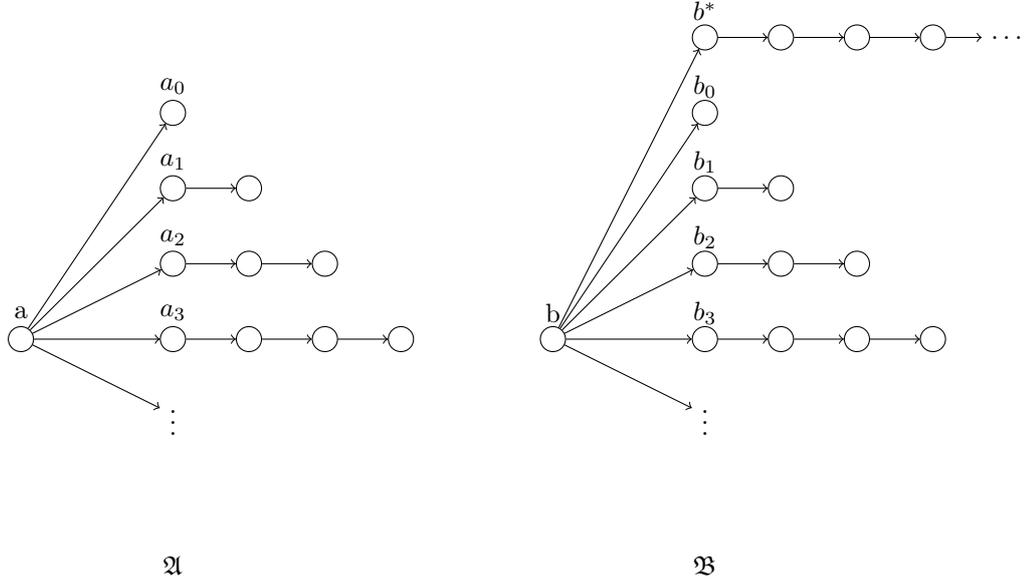
\begin{figure}
	\centering
	\begin{tikzpicture}
		\begin{scope}[every node/.style={circle,draw}]
			\node[draw=none] at (0, -1 + 0.35) {a};
			\node (a) at (0, -1) {};

			\foreach \n in {0,...,3}{
				\node[draw=none] at (2, 2 - \n + 0.35) {$a_{\n}$};
				\node (a\n0) at (2, 2 - \n) {};
				\path [->] (a) edge (a\n0);
				}
			\foreach \n in {1,...,3}{
					\foreach \m in {1,...,\n}{
							\node(a\n\m) at (2 + \m, 2 - \n) {};
							\pgfmathtruncatemacro{\mm}{\m - 1}%
							\path [->] (a\n\mm) edge (a\n\m);
						}
				}
		\end{scope}
		\node (an) at (2, 2 - 4) {\vdots};
		\path [->] (a) edge (an);
		
		\node[] at (2, 2 - 4 - 2) {$\str{A}$};
		
		\begin{scope}[every node/.style={circle,draw}]
			\node[draw=none] at (7, -1 + 0.35) {b};
			\node (b) at (7, -1) {};
			
			\node[draw=none] at (7 + 2, -3 + 0.35 + 6) {$b^{\ast}$};
			\node (bs0) at (7 + 2, -3 + 6) {};

			\path [->] (b) edge (bs0);
			\foreach \m in {1,...,3}{
					\node(bs\m) at (7 + 2 + \m, -3 + 6) {};
					\pgfmathtruncatemacro{\mm}{\m - 1}%
					\path [->] (bs\mm) edge (bs\m);
				}
			
			\foreach \n in {0,...,3}{
				        \node[draw=none] at (7 + 2, 2 - \n + 0.35) {$b_{\n}$};
					\node (b\n0) at (7 + 2, 2 - \n) {};
					\path [->] (b) edge (b\n0);
				}
			\foreach \n in {1,...,3}{
					\foreach \m in {1,...,\n}{
							\node(b\n\m) at (7 + 2 + \m, 2 - \n) {};
							\pgfmathtruncatemacro{\mm}{\m - 1}%
							\path [->] (b\n\mm) edge (b\n\m);
						}
				}
		\end{scope}
		\node (bn) at (7 + 2, 2 - 4) {\vdots};
		\path [->] (b) edge (bn);
		
		\node (bsn) at (7 + 2 + 4, -3 + 6) {\dots};
		\path [->] (bs3) edge (bsn);
		
		\node[] at (7 + 2, 2 - 4 - 2) {$\str{B}$};
	\end{tikzpicture}
	\caption{The structures $(\str{A}, a)$ and $(\str{B}, b)$ have the same $\ML$-theory, but they are not $\ML$-bisimilar.}
	\label{figure:hm-counterexample}
\end{figure}

\paragraph{} To better understand the importance of saturation, consider
the standard example of the two pointed Kripke structures 
in Figure~\ref{figure:hm-counterexample}. They have the same theory w.r.t.\ basic modal logic $\ML \equiv \kl[\set{\Diamond}]$ as \player{2} can win any finite $\kl[\set{\Diamond}]$-bisimulation game; but they
are not $\kl[\set{\Diamond}]$-bisimilar, as \player{1} may choose the witness
$\set{b^{\ast}}$ and is thus able to continue picking successors indefinitely, while \player{2} gets stuck after $n + 1$ rounds if $\set{a_n}$ was selected as the first response.
\par The issue appears to be that $\str{B}$ possesses a richer internal structure than $\str{A}$, and indeed, $\str{B}$ is $\kl[\set{\Diamond}]$-saturated, while $\str{A}$ is not.
Such imbalance cannot occur between two $\kl[\set{\Diamond}]$-saturated structures, and this
phenomenon naturally also adapts to the scenario of $k$-quantifier logics as shown in Theorem~\ref{theorem:hennessy-milner} below.

\par Note that the modal case is special as it just requires
saturation w.r.t.\ $\Diamond$-types since $s$-types are always
realised due to them being essentially (supersets of) singleton sets.
Hennessy--Milner results for neighbourhood logic and 
for inquisitive modal logic, which are closer in
spirit to the present more general setting, are treated in~\cite{Hansen09}
and~\cite{MeissnerOtto},
respectively.

\begin{theorem}[a Hennessy--Milner Theorem]\label{theorem:hennessy-milner}
Let $\kl$ be a k-quantifier logic and $(\str{A}, \alpha)$, $(\str{B}, \beta)$ 
a pair of pointed
	$\signature$-structures. If $(\str{A}, \alpha)$ and $(\str{B}, \beta)$ are
	$\kl$-saturated then
	\begin{align*}
		(\str{A}, \alpha) \equiv_{\kl} (\str{B}, \beta) \quad & \implies \quad (\str{A}, \alpha) \sim_{\kl} (\str{B}, \beta)
.\end{align*}
\end{theorem}                                                              
\begin{proof}
		We show that player~\player{2} can maintain $\kl$-equivalence,
		in response to any challenges by~\player{1}, as an invariant that guarantees
		a win in any infinite play. 
		Let w.l.o.g.\ $s \in \quantifier(\alpha)$ be \player{1}'s challenge in the $\kl$-bisimulation game. Consider
		\begin{align*}
			\Psi & \defas \Bigl\{ \bigvee_{\gamma \in s} \bigwedge \chi_{\str{A}, \gamma}^q(\kl[\quantifiers_0](\signature_0)) \mid q \in \N, \signature_0 \subset \signature \text{ finite}, \quantifiers_0 \subset \quantifiers \text{ finite} \Bigr\}
		\end{align*}
where $\chi_{\str{A}, \gamma}^q(\kl[\quantifiers_0](\signature_0))$ is
the characteristic formula for $(\str{A}, \gamma)$ at quantifier rank
$q$ with respect to the fragment
$\kl[\quantifiers_0](\signature_0)$. Note that the conjunctions and
disjunctions in $\Psi$ are finite up to logical equivalence as we
restrict the language to finite signatures and sets of quantifiers. The
set $\Psi$ is by definition a $\quantifier$-type of $(\str{A},
\alpha)$ and thus of $(\str{B}, \beta)$ as
$(\str{A}, \alpha) \equiv_\LL (\str{B}, \beta)$. 
Since $\str{B}$ is saturated, there is a $t \in
\quantifier_{\str{B}}(\beta)$ realising $\Psi$, i.e.\ $\str{B}, \delta
\models \Psi$ for all $\delta \in t$.
		\par Let $t$ be \player{2}'s response and let $\delta \in t$ be \player{1}'s challenge for the second phase of this round. By choice of $t$, for every $q \in \N$ and finite $\signature_0 \subset \signature$, $\quantifiers_0 \subset \quantifiers$ there is some $\gamma \in s$ such that $\chi_{\str{B}, \delta}^q(\kl[\quantifiers_0](\signature_0)) = \chi_{\str{A}, \gamma}^q(\kl[\quantifiers_0](\signature_0))$. As every finite subset $\Psi_0 \subset \Th{\str{B}, \delta}$ only contains finitely many symbols and quantifiers, it is, by
                Theorem~\ref{theorem:ehrenfeucht-fraisse} and
                Proposition~\ref{proposition:characteristic-formulae}, 
a consequence of
some $\chi_{\str{B}, \delta}^q(\kl[\quantifiers_0](\signature_0))$. Therefore, the theory $\Th{\str{B}, \delta} \equiv \bigcup_{q, \signature_0, \quantifiers_0} \chi_{\str{B}, \delta}^{q}(\kl[\quantifiers_0](\signature_0))$ is an $s$-type of $\str{A}$. As $\str{A}$ is saturated, there must be  some $\gamma \in s$ with $\Th{\str{B}, \delta} = \Th{\str{A}, \gamma}$, which it is safe for \player{2} to play.
	\end{proof}

\section{The Łoś Property and First-Order Logic}\label{section:ultraproducts}
We quickly recall some definitions related to ultraproducts. For a more thorough introduction, we refer to~\cite[Chapter~9.5]{Hodges1993}.
\par A collection $\filter \subset \powset{I} \setminus \{ \emptyset \}$ of nonempty subsets is a \emph{filter} over the index set $I$ if it is closed under passage to supersets (upward closed: $\filter^\uparrow = \filter$) and under finite intersections
($J_1,J_2 \in \filter \Rightarrow J_1\cap J_2 \in \filter$, or $\filter  \cap \filter = \filter$). An \emph{ultrafilter} over $I$ is a filter $\uf$ such that, for every $J \in \powset{I}$, either $J \in \uf$ or $I\setminus J \in \uf$. The latter is a maximality condition w.r.t.\ the $\subset$-relation (over $\powset{\powset{I}}$) among filters, and an application of Zorn's lemma shows that every filter
$\filter$ can be extended to an ultrafilter $\uf \supset \filter$. An ultrafilter extending the filter $\set{\set{i}}$ for some $i \in I$ is called \emph{principal} and \emph{non-principal} otherwise. For $I = \N$, every non-principal ultrafilter contains the Fr\'echet filter, consisting of all co-finite subsets.

The direct product of a family $(\str{A}_i)_{i \in I}$ of
$\sigma$-structures  is denoted as $\prod_{i \in I} \str{A}_i$. We use boldface letters like
$\de{a}$ for its elements $\de{a} = (a_i)_{i \in I} \in \prod_{i \in I} A_i$, and extend this notational convention to tuples of elements $\de{a} =(\de{a}^{(1)}, \ldots, \de{a}^{(n)})
	\in (\prod_{i \in I} A_i)^n$, where we then also write
$\de{a}_i := (a_i^{(1)}, \ldots, a_i^{(n)})
	\in A_i^n$ for the tuple of entries in the $i$th component of the product. 
Direct products over constant families, $\str{A}_i = \str{A}$ for all $i \in I$, 
are denoted $\str{A}^I$ and addressed as direct powers.  
The \emph{$\filter$-reduced product} of the family $(\str{A}_i)$ w.r.t.\ the filter $\filter$ is obtained as a $\sigma$-structure
$\prod \str{A}_i/\filter$
over the quotient set $\prod_{i \in I}A_i/{\sim_\filter}$. The equivalence relation
\[
	\de{a} \sim_\filter \de{a}' \; \mbox{ if } \; \{ i \in I \mid a_i = a_i' \} \in \filter
\]
identifies two elements $\de{a} = (a_i)_{i \in I}$ and $\de{a}' = (a_i')_{i \in I}$
of the direct product precisely if they agree in $\filter$-many components.
We denote $\sim_\filter$-equivalence classes of elements $\de{a}$ as
$\de{a}_\filter \in \prod A_i/{\sim_\filter}$, and again extend this notation
to $n$-tuples.
Note that for $\de{a}^{(1)}, \de{b}^{(1)}, \dots, \de{a}^{(k)}, \de{b}^{(k)} \in \prod A_i$,
\begin{align*}
	(\de{a}^{(1)}, \dots, \de{a}^{(k)}) \sim_\filter (\de{b}^{(1)}, \dots, \de{b}^{(k)})
	\,\iff\, 
	\de{a}^{(1)} \sim_\filter \de{b}^{(1)}, \dots, \de{a}^{(k)} \sim_\filter \de{b}^{(k)}
.\end{align*}
The interpretation of the relation symbols $R \in \sigma$ in
$\str{C} := \prod \str{A}_i/{\sim_\filter}$ is defined according to 
\[
	\de{a} \in R^{\str{C}}  \; \mbox{ if } \; \{ i \in I \mid \de{a}_i \in R^{\str{A}_i} \} \in \filter,
\]
representing membership in $\filter$-many components (which is
well-defined, i.e.\ independent of representatives, 
in the quotient, due to closure of $\filter$ under finite intersections). Notation
like $\str{A}^I/\filter$ or $\str{A}^I/\uf$ correspondingly refers to \emph{reduced powers} or \emph{ultrapowers}, i.e.\ reduced products w.r.t.\ a filter or ultrafilter
for a family of structures with constant entry $\str{A}$ for the component structures $\str{A}_i$. All of these conventions extend naturally to $\sigma$-structures with assignments $\alpha_i := \de{a} \in A_i^n$
as component objects. 
Already for the definition of the quotient and the interpretation of the relation symbols it is suggestive to evaluate formulae over the direct product
$\prod_{i \in I} (\str{A}_i,\alpha_i)$, with assignments
$\alpha = \de{a} \in (\prod_{i \in I} A_i)^n$,
in the boolean powerset algebra $\powset{I}$, giving 
\[
	\brck{\phi}_\alpha := \{ i \in I \mid \str{A}_i,\alpha_i\models \phi \} \in \powset{I}
\]
as a ``truth value'' to $\phi$ over $\prod_{i \in I}(\str{A}_i,\alpha_i)$.
Then the interpretation of
the relations in the reduced product $\prod \str{A}_i/\filter$ amounts to
the following equivalence for all atomic formulae $\phi$:
\[
	\prod \str{A}_i/\filter,\alpha_\filter \models \phi
	\quad
	\Longleftrightarrow \quad
	\brck{\phi} \in \filter.
\]
The classical theorem of \L o\'s states precisely this equivalence for all
$\phi \in \FO$ in the case of ultraproducts.

\begin{definition}[Łoś property]\label{definition:los}
	We say that an abstract logic $\al$ has the \emph{Łoś property} if for every pointed ultraproduct $(\prod \str{A}_i/\uf, \alpha_\uf)$ and formula $\phi \in \al(\signature)$,
	\begin{align*}
		\prod \str{A}_i / \uf, \alpha_{\uf} \models \phi & \iff
		\brck{\phi}_\alpha \in \uf 
		.\end{align*}
\end{definition}

\begin{theorem}[Łoś {\cite[Theorem~9.5.1]{Hodges1993}}]\label{theorem:los}
	First-order logic has the Łoś property.
\end{theorem}

\begin{definition}[compactness property]\label{definition:compactness}
	We say that an abstract logic $\al$ has the \emph{compactness property} if every \emph{finitely satisfiable} set of formulae is \emph{satisfiable}, i.e.\ for every set of formulae $\Phi \subset \al(\signature)$ the following are equivalent:
	\begin{enumerate}[label=(\roman*)]
		\item\label{prop:los-implies-compactness:i} $\Phi$ has a model $\str{A}, \alpha \models \Phi$, and
		\item\label{prop:los-implies-compactness:ii} every finite $\Phi_0 \subset \Phi$ has a model $\str{A}_0, \alpha_0 \models \Phi_0$.
	\end{enumerate}
\end{definition}

It is a well-known fact that for abstract logics the Łoś property implies compactness (see, e.g., \cite[Corollary~4.1.11]{ChangKeisler1990} for the $\FO$ version).

\paragraph{} There is a set-theoretically much deeper result due to Keisler and Shelah stating that two structures are indistinguishable by first-order formulae if, and only if, they possess isomorphic ultrapowers~\cite[Theorem~6.1.15]{ChangKeisler1990}.
\par As a corollary to the theorem,
having the Łoś property is equivalent to being a fragment of
first-order logic.
Put differently, in the context of Definition~\ref{definition:abstract-logics},
first-order logic is the global
maximum among all abstract logics having the Łoś property.

\begin{lemma}\label{lemma:growing-los}
	Let $\al_1$ and $\al_2$ be two abstract logics that have the
        Łoś property and $\al$ the \emph{$\FO$-closure} of $\al_1 \cup
        \al_2$, i.e.\ $\al(\signature)$ is
the closure of $\al_1(\signature) \cup \al_2(\signature)$ under
boolean connectives, universal and existential quantification.
Then $\al$ also has the Łoś property.
\begin{proof}[Proof by syntactic induction] The proof is very similar
  to the usual proof of the Łoś property for first-order logic. The
  base case, where $\phi \in \al_1 \cup \al_2$, is given by
  assumption. For the induction step, we use the maximality of $\uf$
  to show compatibility with negation and disjunction,
  while closure under finite intersections and supersets takes care of conjunctions. The step for the existential quantifier also works as usual.
	\end{proof}
\end{lemma}

\begin{lemma}\label{lemma:disagreeing-models-one-step}
	Let $\hat{\al}$ be an abstract logic extending the abstract logic $\al$. Furthermore, let $\phi \in \al(\signature)$, $\psi \in \hat\al(\signature)$ and $T \subset \al(\signature)$. If for all $\chi \in \al(\signature)$, $T \nmodels \psi \liff \chi$ then
	\begin{align*}
		T \cup \set{\phi}                        \nmodels \psi \liff \chi & \quad \text{for all } \chi \in \al(\signature)\phantom{.} \\
		\text{or } \qquad T \cup \set{\lnot \phi} \nmodels \psi \liff \chi & \quad \text{for all } \chi \in \al(\signature).
	\end{align*}
	\begin{proof}
		Towards a contradiction, assume there exist formulae $\chi$ and $\chi^{\prime}$ such that $T \models \phi \limplies (\psi \liff \chi)$ and $T \models \lnot \phi \limplies (\psi \liff \chi^{\prime})$. Then $T \models \eta$ for $\eta$ given as $\eta \defas \psi \liff ((\phi \land \chi) \lor (\lnot \phi \land \chi^{\prime}))$, contradicting our assumption as $\eta$ is in $\al(\signature)$.
	\end{proof}
\end{lemma}

\begin{lemma}\label{lemma:disagreeing-models}
	Let $\hat{\al}$ be a compact abstract logic extending an abstract logic $\al$ and $\phi \in \hat{\al} \setminus \al$ not expressible in $\al$. Then there are pointed $\signature$-structures $(\str{A}, \alpha)$ and $(\str{B}, \beta)$ with
	\begin{align*}
		\Th<\al>{\str{A}, \alpha} & = \Th<\al>{\str{B}, \beta}
		,\end{align*}
	while $\str{A}, \alpha \models_{\hat{\al}} \phi$ and
	$\str{B}, \beta \models_{\hat{\al}} \lnot \phi$.
\end{lemma}

\begin{proof}
  Compactness of $\hat{\al}$ ensures that, in the situation of
  Lemma~\ref{lemma:disagreeing-models-one-step}, the sets $T$ are closed under unions of $\subset$-chains. Applying Zorn's Lemma yields a $\subset$-maximal such $T$. Lemma~\ref{lemma:disagreeing-models-one-step} ensures that $T$ is complete. Furthermore, $T \cup \set{\phi}$ as well as $T \cup \set{\lnot \phi}$ are satisfiable, as $\phi$ would otherwise be equivalent to $\ltrue$ or $\lfalse$ under $T$.
\end{proof}

Combining these observations, 
we obtain the following corollary to the theorem by Keisler and Shelah.

\begin{corollary}\label{corollary:los-fo}
	Let $\al$ be an abstract logic. Then the following are equivalent:
	\begin{enumerate}[label=(\roman*)]
		\item\label{corollary:los-fo:1} $\al$ has the Łoś property.
		\item\label{corollary:los-fo:2} $\al \preceq \FO$.
	\end{enumerate}
\end{corollary}

\begin{proof}
	That \ref{corollary:los-fo:2} implies~\ref{corollary:los-fo:1} is obvious. For  \ref{corollary:los-fo:1} $\Rightarrow$~\ref{corollary:los-fo:2} we give
	an indirect argument.
	Towards a contradiction, assume that $\al$ has the Łoś
        property, but there is some $\phi \in \al(\signature)$ that is
        not expressible in $\FO$. Let $\hat{\LL}$ be the
        $\FO$-closure of $\al \cup \FO$. Then, by Lemma~\ref{lemma:growing-los}, $\hat{\LL}$ also has the Łoś property and is thus compact. Therefore, we can apply Lemma~\ref{lemma:disagreeing-models} to find pointed $\signature$-structures $(\str{A}, \alpha)$ and $(\str{B}, \beta)$ with $\Th<\FO>{\str{A}, \alpha} = \Th<\FO>{\str{B}, \beta}$, while $\str{A}, \alpha \models \phi$ and $\str{B}, \beta \models \lnot \phi$.
	\par As $(\str{A}, \alpha)$ and $(\str{B}, \beta)$ are indistinguishable for $\FO$, we may apply Keisler--Shelah to obtain two isomorphic ultrapowers $(\prod \str{A}/\uf, \alpha_{\uf}) \isom (\prod \str{B}/\uf, \beta_{\uf})$ over $\str{A}$ and $\str{B}$, respectively. The Łoś property of $\al$ ensures that ${\prod \str{A}/\uf, \alpha_{\uf} \models \phi}$, while ${\prod \str{B}/\uf, \beta_{\uf} \models \lnot \phi}$, contradicting the isomorphism invariance of $\al$. Thus, $\al$ has to be a fragment of $\FO$.
\end{proof}

For k-quantifier logics, we obtain an even stronger result in the sense that there is a uniform translation of every quantifier into first-order logic if the corresponding k-quantifier logic has the Łoś property.

\begin{theorem}[uniform $\FO$-translation]\label{theorem:uniform-translations}
	Let $\kl$ be a $k$-quantifier logic
        that has the Łoś property. 
	Then for every quantifier $\quantifier \in \quantifiers$
	and $k$-ary relation symbol $R \notin \signature_{\quantifier}$
	there is a formula $\eta \in \FO(\signature_{\quantifiers} \cup \set{R})$ 
	with $k$ free variables such that for all signatures 
	$\signature \supset \signature_{\quantifier}$ with $R \notin \signature$, 
	pointed $\signature$-structures $(\str{A}, \alpha)$ and $\phi \in \kl(\signature)$
	\begin{align*}
		\str{A}, \alpha \models_{\kl} \quantifier \phi \quad & \iff \quad \str{A}, \alpha \models_{\FO} \eta[\phi^{\ast}/R]
	\end{align*}
	where $\phi^{\ast} \in \FO(\signature)$ is a translation of $\phi$ to first-order logic.
	\par In particular, this yields a \emph{uniform} translation from $\kl$ into $\FO$.
	\begin{proof}
		By Corollary~\ref{corollary:los-fo}, we know that $\kl$ is a
                fragment of $\FO$. Thus, let ${\eta \in
                  \FO(\signature_{\quantifier} \cup \set{R})}$ be a
                translation of
		$\quantifier R(x_1,...,x_k) \in \kl(\signature_{\quantifier} \cup
                \set{R})$ into $\FO$.
		\par Take any $\signature \supset \signature_{\quantifier}$, a
		fresh relation symbol ${R^{\prime}} \notin \signature$, $\phi \in \kl(\signature)$ with translation $\phi^{\ast} \in \FO(\signature)$ and a pointed $\signature$-structure $(\str{A}, \alpha)$. Define
		\begin{align*}
			{R^{\prime}}^{\str{A}} \defas R^{\str{A} \restriction \signature_{\quantifier}} \defas \set{\gamma \in A^k}[\str{A}, \gamma \models \phi].\end{align*}
		Then, as k-quantifier logics are well-behaved w.r.t.\ expansions,
		\begin{align*}
			\str{A}, \alpha \models_{\kl} \quantifier \phi
			 			 & \iff \str{A}, \alpha, {R^{\prime}}^{\str{A}} \models_{\kl} \quantifier {R^{\prime}}(x_1, \dots, x_k)                                                                                        \\
			 & \iff \str{A} \restriction \signature_{\quantifier}, \alpha, R^{\str{A} \restriction \signature_{\quantifier}} \models_{\kl} \quantifier R(x_1, \dots, x_k)                                  \\
			 & \iff \str{A} \restriction \signature_{\quantifier}, \alpha, R^{\str{A} \restriction \signature_{\quantifier}} \models_{\FO} \eta.                                                           \\
			\intertext{First-order logic is invariant under renaming of relations and well-behaved w.r.t.\ expansions. Thus, we obtain}
			\str{A}, \alpha \models_{\kl} \quantifier \phi
			 & \iff \str{A} \restriction \signature_{\quantifier}, \alpha, {R^{\prime}}^{\str{A}} \models_{\FO} \eta[{R^{\prime}}/R]                                                                       \\
			 & \iff \str{A}, \alpha \models_{\FO} \eta[\phi^{\ast}/R]
			.\qedhere\end{align*}
	\end{proof}
\end{theorem}

\section{A Lindström theorem}\label{section:lindstroem}

In this section, we establish the maximality of $k$-quantifier logics that have the Łoś property under the corresponding notion of bisimulation invariance and the Łoś property itself. Crucially, we do \emph{not} refer to Keisler--Shelah for the proof, but make do with much more elementary arguments.

\begin{lemma}[lifting witnesses]\label{lemma:lifting-witnesses}
	Let $\kl$ be a $k$-quantifier logic that
	has the Łoś property, $\quantifier \in \quantifiers$, and $(\str{A} , \alpha_{\uf}) \defas (\prod \str{A}_i/\uf, \alpha_{\uf})$
	a $k$-pointed ultraproduct based on
	$(\str{A}_i)_{i \in I}$ and ultrafilter $\uf$ over~$I$. Then for any
	$J \in \uf$ and family of witnesses $\collection{s_i}[i \in J]$
	with $s_i \in \quantifier(\str{A}_i, \alpha_i)$, there is an $s
        \in \quantifier(\str{A}, \alpha_{\uf})$
	and $K \in \uf$ such that 
	\[
		s \subset \{ \gamma_\uf \mid \gamma_i \in s_i \mbox{ for all } i \in J \cap K \}
	.\]
\end{lemma}

\begin{proof}
	Consider for all $i \in I$ expansions $\str{A}_i^{\ast}$ of $\str{A}_i$ by a fresh relation $R \notin \signature$ of arity $k$ via
	\begin{align*}
		R^{\str{A}_i^{\ast}} & \defas \begin{cases}
			                              s_i & \text{if } i \in J, \\
			                              \emptyset                                                     & \text{otherwise}
			                              .\end{cases}
	\end{align*}
	The ultraproduct $\str{A}^{\ast} \defas \prod_{i \in I}
        \str{A}_i^{\ast} / \uf$ is an expansion of $\str{A}$ by $R$
        and, by definition of k-quantifiers
        (Definition~\ref{definition:k-quantifiers}), the witness sets
        $\quantifier(\str{A}, \alpha_{\uf})$ and
        $\quantifier(\str{A}^{\ast}, \alpha_{\uf})$ are
        identical. Thus, it suffices to find a suitable $s \in
        \quantifier(\str{A}^{\ast}/\uf, \alpha_{\uf})$.

        \par By construction, $\str{A}_i^{\ast}, \alpha_i \models
		\quantifier R(x_1, \dots, x_k)$ is witnessed by $s_i \in
		\quantifier(\str{A}_i, \alpha_i)$ for all $i \in J$. The Łoś property of
	$\kl$ thus guarantees that
	${\str{A}^{\ast}, \alpha_{\uf} \models \quantifier R(x_1,
          \dots, x_k)}$, witnessed by some witness $s \in
        \quantifier(\str{A}, \alpha_{\uf})
=  \quantifier(\str{A}^\ast, \alpha_{\uf}) $
        with ${s \subset R^{\str{A}^{\ast}}}$. Using the Łoś property, we find a
	$K \in \uf$ s.t.\
	\begin{align*}
		s \subset R^{\str{A}^{\ast}} &= \{ \gamma_\uf \mid \gamma_i \in s_i \mbox{ for all } i \in J \cap K \}
	.\qedhere\end{align*}
\end{proof}

By combining our previous insights, we are now in a position to show the existence of $\kl$-saturated structures by rather pedestrian arguments in the case of countable $\sigma$ and $\quantifiers$.

\begin{lemma}[saturated extensions]\label{lemma:finding-saturated-extensions}
	Let $\LL := \kl$ be a k-quantifier logic with countably many $k$-quantifiers that has the Łoś property
        and $\str{A}$ a $\signature$-structure
for countable $\signature$. Then for every non-principal ultrafilter $\uf$ over $\N$,
$\str{A}^{\ast} \defas \str{A}^\N / \uf$ 
is an $\LL$-saturated $\LL$-elementary extension of $\str{A}$.
	\begin{proof}
 Using the Łoś property, it is easy to verify that $\str{A}^{\ast}$ is an $\LL$-elementary extension via the embedding $\iota: a \mapsto [(a)_{i \in I}]_{\uf}$.
 \par
		Let $\Psi$ be any  $Q$-type of $(\str{A}^{\ast},
                \alpha_{\uf})$. As $\LL(\signature)$ and thus $\Psi
                \subset \LL(\signature)$ is countable, we may
                enumerate it as $(\delta_j)_{j \in \N}$ and consider
                for every $k \in \N$ the set $J_k \defas \brck{Q
                  (\delta_0 \land \dots \land \delta_k)} \cap \set{i
                  \in I}[i \geq k]$.
All the $J_k$ are in $\uf$ as
                $\Psi$ is a $Q$-type of $(\str{A}, \alpha_{\uf})$, since 
                $\uf$ extends the Fr\'echet filter and filters are
                closed under finite intersections. Furthermore, the
                sequence $(J_k)_{k \in \N}$ is decreasing with
                $\bigcap J_k = \emptyset$. For each $i \in \N$, we set
                $k_i \defas \max \set{k \in \N}[i \in J_k]$ and pick a
                witness $s_i$ for $(\str{A}, \alpha_i) \models Q
                (\delta_0 \land \dots \land \delta_{k_i})$. 
By Lemma~\ref{lemma:lifting-witnesses}, we know that there is a $t_{\uf} \in Q(\alpha_{\uf})$
		and $K \in \uf$ s.t.\ $t_{\uf} \subset \{ \gamma_{\uf}
                \mid \gamma_i \in s_i \text{f.a.\ } i \in K \}$.
                Take any $\delta_k \in \Psi$ and $\gamma_{\uf} \in t_{\uf}$. Then, by construction, for all $i \in J_{k_i} \cap K$, $(\str{A}, \gamma_i) \models \delta_k$. As $J_{k_i} \cap K \in \uf$, $t_{\uf}$ realises $\Psi$.
		\par Similarly, let $\Psi = \set{ \delta_j }[ j \in \N ]$ be an $s_{\uf}$-type of $\str{A}^{\ast}$ and consider for all $k \in \N$,
		\begin{align*}
			J_k \defas \set{i \in \N}[\text{ex.\ } \gamma_i \in s_i \text{ s.t.\ } A, \gamma_i \models \delta_0 \land \dots \land \delta_k] \cap \set{i \in \N}[i \geq k]
		.\end{align*}
For $i \in \N$, set $k_i \defas \max \set{k \in \N}[i \in J_k]$ and iteratively
pick $\gamma_i \in s_i$ for which ${\str{A}, \gamma_i \models \delta_0 \land \dots \land \delta_{k_i}}$. Then $\gamma_{\uf}$ realises $s_{\uf}$ as desired.
	\end{proof}
\end{lemma}

Towards the maximality claim in our Lindstr\"om result, we use the 
natural notion of $\sim_{\kl}$-invariance.

\begin{definition}[$\sim_{\kl}$-invariance]\label{definition:simklinv} 
	Let $\al$ be an abstract $k$-logic and $\kl$ a $k$-quantifier logic. A formula $\phi \in \al(\signature)$ is \emph{$\sim_{\kl}$-invariant} if for all $k$-pointed $\signature$-structures $(\str{A}, \alpha), (\str{B}, \beta)$,
	\begin{align*}
		(\str{A}, \alpha) \sim_{\kl} (\str{B}, \beta)
		\quad &\implies \quad
		(\str{A}, \alpha \models \phi \, \iff \, \str{B}, \beta \models \phi)
                        .\end{align*}
	An abstract $k$-logic $\LL$ is \emph{$\sim_{\kl}$-invariant} if all its
formulae are, i.e.\ $\sim_{\kl}$ implies $\equiv_{\LL}$.
\end{definition}

Finally, we obtain a Lindström theorem that characterises
$\sim_{\kl}$-invariant
$k$-quantifier logics that have the Łoś property.

\begin{theorem}[a Lindström theorem]\label{theorem:lindström}
	Let $\kl$ be a $k$-quantifier logic that has the Łoś property. Then every $\kl$-bisimulation invariant abstract
	$k$-logic $\al$ that has the Łoś property is a fragment of
	$\kl$.
\end{theorem} 

\begin{proof}
	Let $\al$ be such an abstract $k$-logic with $\al \not\preceq \kl$ and $\phi \in \al(\signature)$ be any formula 
	not expressible in $\kl$. Since $\al$ has the finite occurrence property, we may assume that $\kl(\signature)$ is countable as $\quantifiers$ is countable by definition and $\phi \in \al(\signature_0)$ for some finite $\signature_0 \subset \signature$.
	Define the logic $\hat\al$ as the $\FO$-closure of $\al \cup \kl$. By Lemma~\ref{lemma:growing-los}, $\hat\al$ has the Łoś property and is thus compact. Hence, we can apply Lemma~\ref{lemma:disagreeing-models} to obtain two pointed $\signature$-structures $(\str{A}, \alpha)$ and $(\str{B}, \beta)$ that have the same $\kl$-theory, while $\str{A}, \alpha \models \phi$ and $\str{B}, \beta \models \lnot \phi$.
	
	Using Lemma~\ref{lemma:finding-saturated-extensions}, we find $\kl$-saturated
	$\al$-elementary extensions $\str{A}^{\ast}$ and $\str{B}^{\ast}$ of
	$\str{A}$ and $\str{B}$, respectively.
	As $\str{A}^{\ast}$ and $\str{B}^{\ast}$ are $\kl$-saturated and
	$(\str{A}^{\ast}, \alpha)$ and $(\str{B}^{\ast}, \beta)$ have the same
	$\kl$-theory, Theorem~\ref{theorem:hennessy-milner} implies that the pointed
	structures are $\kl$-bisimilar, so that $\phi \in \al$ would violate 
	the $\kl$-bisimulation invariance of $\al$ as $\str{A}^{\ast}, \alpha \models \phi$, while $\str{B}^{\ast}, \beta \models \lnot \phi$.
\end{proof}
\par By Corollary~\ref{corollary:los-fo}, there is essentially no difference between
being a fragment of $\FO$ and having the Łoś property. Thus, every
Lindström theorem w.r.t.\ the Łoś property is also a characterisation
theorem, and we obtain the following different perspective on Theorem~\ref{theorem:lindström}.

\begin{corollary}[characterisation theorem]
	Let $\kl$ be a $k$-quantifier logic that has the Łoś property. Then any $\sim_{\kl}$-invariant formula of $\FO^k$ is equivalently expressible in $\kl$.
\end{corollary}

\section{Summary and Outlook}
This note suggests a view of a certain kind of quantification that is
limited to dealing with $k$-tuples of elements but takes into account 
a built-in notion of accessibility of families of sets of $k$-tuples from given
$k$-tuples. The underlying pattern is thus reminiscent of modal
scenarios (broadly conceived) with a twist towards the involvement of
second-order objects. While similar notions have been explored in
neighbourhood semantics as well as in inquisitive modal logic, our
$k$-quantifiers are in a sense more general in terms of the actual
semantics of concrete quantifiers -- but also more rigid in terms of the
semantics of each individual $k$-quantifier as the isomorphism
type of the underlying plain first-order structure fully determines the
accessibility pattern that defines its semantics over that
structure. The Ehrenfeucht--Fra\"\i ss\'e style notions of back\& forth
equivalence associated with a fixed supply of such quantifiers 
works out naturally as expected. For the associated 
notions of saturation and the natural Hennessy--Milner theorem to go
with it, we here resort to the very strong assumption of compatibility with
ultraproducts. By a much deeper theorem of Keisler and Shelah this
assumption of course actually ties the expressive power of the logics
under consideration to within first-order. Our Lindstr\"om result --
as is the main target of these investigations --  for such $k$-quantifier
logics in the context of abstract $k$-logics is, however,
established on a much more pedestrian route that
illustrates the natural involvement of compactness arguments and
the r\^ole of suitable saturation properties in a much more explicit manner. 
It remains open to which extent weaker assumptions, like, e.g.\ compactness
together with suitable adaptations of a Tarski union property, could
be applied to a similar effect against the extremely general background
assumptions about abstract $k$-logics proposed here. Another promising
direction for related future research could therefore also concern 
meaningful limitations of this generality, maybe towards a focus that
brings us closer again to more familiar scenarios like neighbourhood
or inquisitive semantics.

\bibliographystyle{plain}

\bibliography{references}

\end{document}